\documentclass[12pt]{article}

\usepackage{amsthm,amsmath,natbib}

\usepackage{amsfonts}
\usepackage{graphicx}
\usepackage{comment}

\newtheorem{theorem}{Theorem}[section]
\newtheorem{lemma}[theorem]{Lemma}
\newtheorem{proposition}[theorem]{Proposition}
\newtheorem{corollary}[theorem]{Corollary}

\begin{document}

\def\spacingset#1{\renewcommand{\baselinestretch}%
{#1}\small\normalsize} \spacingset{1}

  \title{\bf Lattice-based designs possessing quasi-optimal separation distance on all projections}
  \author{Xu He\hspace{.2cm}\\
    Academy of Mathematics and System Sciences, \\Chinese Academy of Sciences}
  \maketitle

\begin{abstract}
Experimental designs that spread out points apart from each other on projections are important for computer experiments when not necessarily all factors have substantial influence on the response. 
We provide a theoretical framework to generate designs that possess quasi-optimal separation distance on all of the projections and quasi-optimal fill distance on univariate margins. 
The key is to use special techniques to rotate certain lattices. 
One such type of design is densest packing-based maximum projection designs, 
which outperform existing types of space-filling designs in many scenarios. 
Computer code to generate these designs is provided in R package LatticeDesign. 
\end{abstract}

{\it Keywords:}  Design of experiment, Emulation, Gaussian process model, Geometry of numbers, Latin hypercube, Maximin distance design, Maximum projection design, Minkowski's first theorem.

\spacingset{1.45} 

\section{Introduction}
\label{sec:intro}

Computer experiments have become powerful tools in science and engineering~\citep{Santner:book}. 
The separation distance of an experimental design $D \subset [0,1]^p$ is its minimal pairwise distance, 
\[ \rho(D) = \min_{x,y \in D} \left\{ \sum_{k=1}^p (x_k-y_k)^2 \right\}^{1/2}. \]
Designs that achieve the highest separation distance are called maximin distance designs~\citep{Johnson:1990}. 
Being asymptotically D-optimal in Gaussian process emulation, excellent in controlling numeric error, and robust against simulation errors, 
maximin distance designs are a popular type of space-filling design that are useful for computer experiments~\citep{Johnson:1990,Siem:2007,Haaland:2018,Wang:2018}. 

In many circumstances, not all factors have substantial influence on the response. 
In this context, it is desirable that the design be space-filling after projected onto the subspace for active variables. 
Formally, the separation distance of a design $D \subset [0,1]^p$ projected onto dimensions on $\gamma \subset \{1,\ldots,p\}$ is given by 
\[ \rho_\gamma(D) = \min_{x,y \in D} \left\{ \sum_{k \in \gamma} (x_k-y_k)^2 \right\}^{1/2}. \]
We remark that $\rho(D) = \rho_{\{1,\ldots,p\}}(D)$ under this notation.  
Since the active variables are usually not known before experimentation, an ideal design shall possess high $\rho_\gamma(D)$ for every $\emptyset\neq\gamma\subset\{1,\ldots,p\}$. 
Methods that produce designs with high $\rho_\gamma(D)$ not only for $\gamma=\{1,\ldots,p\}$ but also for its subsets include maximin distance Latin hypercube designs~\citep{Morris:1995}, maximum projection designs~\citep{Roshan:2015}, rotated sphere packing designs~\citep{RSPD}, and quasi-Monte Carlo sequences such as scrambled nets~\citep{Owen:1997}. 

In this paper, we propose a new class of space-filling design called densest packing-based maximum projection designs. 
These designs achieve the asymptotically optimal order of $\rho_\gamma(D)$ for all $\emptyset \neq \gamma \subset \{1,\ldots,p\}$, 
i.e., there exists a $c>0$ such that 
\[ \rho_\gamma(D) \geq c n^{-1/\text{card}(\gamma)}, \]
where $\text{card}(\gamma)$ denotes the cardinality of $\gamma$ and $n=\text{card}(D)$. 
To the best of our knowledge, when $p>2$ no known type of designs possesses this property. 
Furthermore, they also possess the optimal order of fill distance on univariate projections, 
i.e., there exists a constant $d$ such that for any $k \in \{1,\ldots,p\}$, 
\[ \eta_k(D) = \sup_{y \in [0,1]} \min_{(x_1,\ldots,x_p) \in D} |x_k-y| \leq d n^{-1}. \]

In constructing densest packing-based maximum projection designs, 
we adopt the same procedure as in generating rotated sphere packing designs,
which consists of generating a number of rotation matrices, obtaining one lattice-based design from each rotation matrix, and selecting the design with the best empirical projective separation distance. 
From this procedure, projective properties of the generated designs are largely determined by the rotation matrix. 
Lacking theoretical results, 
for $p>2$, original rotated sphere packing designs are constructed from random rotation matrices and thus they have no known asymptotic result on projective property. 
The major contribution of this work is analytically deriving the asymptotic properties of a list of newly proposed rotation matrices. 
We then use some of the proposed matrices to generate densest packing-based maximum projection designs and show that they in general have better projective and unprojected separation distances than original rotated sphere packing designs and other types of designs. 
Our method applies to arbitrary $p$ and $n$, but is most advantageous for $p=4$, $8$, and $16$.

\section{Theoretical results}
\label{sec:theory}

In this section, we review the definition of rotated sphere packing designs and derive theoretical results on designs constructed from certain rotation matrices. 
A set of points in $\mathbb{R}^p$ is called a lattice and written as $L(G)$ if $G$ is a non-singular $p\times p$ real matrix and 
$ L = \left\{ a^T G : a \in \mathbb{Z}^p \right\} $. 
For instance, $\mathbb{Z}^p = L(I_p)$ is called the $p$-dimensional integer lattice, where $I_p$ is the $p$-dimensional identity matrix. 
When $R$ is a $p \times p$ orthogonal matrix, $L(GR)$ is a rotation of $L(G)$. 
Thus, such an $R$ is also referred to as a rotation matrix. 

A lattice-based design is a set of points that can be expressed as the intersection of a rotated, rescaled, and translated lattice and the design space $[0,1]^p$. 
In particular, a rotated sphere packing design in $p$ dimensions and with $n$ points is the design constructed by the following three steps~\citep{RSPD}: 

\begin{enumerate}
\item Choose a $G$. 
\item Generate 100 $R$s; for each $R$, find a $\delta$ such that 
\begin{equation}\label{eqn:D}
 D = D(G,R,n,\delta) = \left\{ \{n|\det(G)|\}^{-1/p} (a^T G R + \delta^T) : a \in\mathbb{Z}^p \right\} \cap [0,1]^p.  
\end{equation}
has exactly $n$ points. 
Such a $\delta$ always exists~\citep{RSPD}. 
\item Select the design from the 100 generated designs that have the best empirical projective distance property, 
measured by the maximum projection criterion proposed in~\citet{Roshan:2015}, 
\begin{equation}
\label{eqn:MPC}
\psi(D) = \left[ \{n(n-1)\}^{-1}
\sum_{1\leq i<j\leq n} \frac{1}{\prod_{k=1}^p (x_{i,k}-x_{j,k})^2} \right]^{1/p}.  
\end{equation} 
\end{enumerate}

For any given $G$, there exists a $c(G)$ such that $\rho(D) = n^{-1/p} c(G)$ for any $D$ in (\ref{eqn:D}). 
That is, all types of rotated sphere packing designs possess the optimal order of unprojected separation distance and the constant is determined solely by the lattice used. 
On the other hand, projective separation distances of the designs are determined by $G$ and $R$. 
In particular, designs with asymptotically optimal order of univariate projective separation and fill distances are called quasi-Latin hypercube designs.  
If designs constructed from $G$ and $R$ are proven to be quasi-Latin hypercube designs,  
we call the $R$ a magic rotation for $G$~\citep{RSPD}. 
The magic rotations we are going to propose are beyond ``magic'' because designs constructed from them also possess multivariate properties. 
Before this work, only one magic rotation matrix in $p=2$ has been found~\citep{RSPD}. 

In the rest of this section we provide magic rotations for four useful types of lattices: 
integer lattices, densest packings, thinnest coverings, and interleaved lattices. 
The $L(G)$ in $p$-dimensions that has the highest $n^{1/p} \rho(D)$ for $D$ in (\ref{eqn:D}) is called the densest packing in $p$ dimensions. 
Let 
\begin{equation}\label{eqn:DP24}
G_{\text{DP},4} = \left(\begin{array}{cccc}
1 & 0 & 0 & 1 \\
0 & 1 & 0 & 1 \\
0 & 0 & 1 & 1 \\
0 & 0 & 0 & 2
\end{array}\right) ,
G_{\text{DP},8} = \left(\begin{array}{cccccccc}
1 & 0 & 0 & 0 & 1 & 1 & 1 & 0 \\
0 & 1 & 0 & 0 & 1 & 1 & 0 & 1 \\
0 & 0 & 1 & 0 & 1 & 0 & 1 & 1 \\
0 & 0 & 0 & 1 & 0 & 1 & 1 & 1 \\
0 & 0 & 0 & 0 & 2 & 0 & 0 & 0 \\
0 & 0 & 0 & 0 & 0 & 2 & 0 & 0 \\
0 & 0 & 0 & 0 & 0 & 0 & 2 & 0 \\
0 & 0 & 0 & 0 & 0 & 0 & 0 & 2
\end{array}\right)  ,
\end{equation}
\begin{equation}\label{eqn:DP816}
\quad\text{and }
G_{\text{DP},16} = \left(\begin{array}{cccccccccccccccc}
4 & 0 & 0 & 0 & 0 & 0 & 0 & 0 & 0 & 0 & 0 & 0 & 0 & 0 & 0 & 0 \\
2 & 2 & 0 & 0 & 0 & 0 & 0 & 0 & 0 & 0 & 0 & 0 & 0 & 0 & 0 & 0 \\
2 & 0 & 2 & 0 & 0 & 0 & 0 & 0 & 0 & 0 & 0 & 0 & 0 & 0 & 0 & 0 \\
2 & 0 & 0 & 2 & 0 & 0 & 0 & 0 & 0 & 0 & 0 & 0 & 0 & 0 & 0 & 0 \\
2 & 0 & 0 & 0 & 2 & 0 & 0 & 0 & 0 & 0 & 0 & 0 & 0 & 0 & 0 & 0 \\
2 & 0 & 0 & 0 & 0 & 2 & 0 & 0 & 0 & 0 & 0 & 0 & 0 & 0 & 0 & 0 \\
2 & 0 & 0 & 0 & 0 & 0 & 2 & 0 & 0 & 0 & 0 & 0 & 0 & 0 & 0 & 0 \\
2 & 0 & 0 & 0 & 0 & 0 & 0 & 2 & 0 & 0 & 0 & 0 & 0 & 0 & 0 & 0 \\
2 & 0 & 0 & 0 & 0 & 0 & 0 & 0 & 2 & 0 & 0 & 0 & 0 & 0 & 0 & 0 \\
2 & 0 & 0 & 0 & 0 & 0 & 0 & 0 & 0 & 2 & 0 & 0 & 0 & 0 & 0 & 0 \\
2 & 0 & 0 & 0 & 0 & 0 & 0 & 0 & 0 & 0 & 2 & 0 & 0 & 0 & 0 & 0 \\
1 & 1 & 1 & 1 & 0 & 1 & 0 & 1 & 1 & 0 & 0 & 1 & 0 & 0 & 0 & 0 \\
0 & 1 & 1 & 1 & 1 & 0 & 1 & 0 & 1 & 1 & 0 & 0 & 1 & 0 & 0 & 0 \\
0 & 0 & 1 & 1 & 1 & 1 & 0 & 1 & 0 & 1 & 1 & 0 & 0 & 1 & 0 & 0 \\
0 & 0 & 0 & 1 & 1 & 1 & 1 & 0 & 1 & 0 & 1 & 1 & 0 & 0 & 1 & 0 \\
1 & 1 & 1 & 1 & 1 & 1 & 1 & 1 & 1 & 1 & 1 & 1 & 1 & 1 & 1 & 1 
\end{array}\right) .
\end{equation}
Then $L(G_{\text{DP},4})$, $L(G_{\text{DP},8})$, and $L(G_{\text{DP},16})$ are the densest packings in four, eight, and sixteen dimensions, respectively. 
Similarly, For any given $G$ and $\epsilon>0$, there exists a $d(G)$ such that for any $n>\epsilon^{-n}$ and any $y \in [n^{-1/p},1-n^{-1/p}]^p$, there exists an $x \in D(G,R,n,\delta)$ such that $\|x-y\| \leq d(G) n^{-1/p}$. 
The $L(G)$ that is associated with the lowest $d(G)$ is called the thinnest covering. 
Let 
\begin{equation}\label{eqn:TC}
 G_{\text{TC},p} = \left\{ (p+1)-(p+1)^{1/2} \right\}I_p - J_p,  
\end{equation}
where $J_p$ denotes the $p\times p$ matrix with all of the entries being 1,  
then $L(G_{\text{TC},p})$ is the $p$-dimensional thinnest covering for $2\leq p\leq 22$.
See~\citet{Conway:1998} for a comprehensive review on densest packings and thinnest coverings. 
Finally, $L(G)$ is called a standard interleaved lattice if $L(2I_p) \subset L(G) \subset L(I_p)$. 
Most designs in finite $n$ that have by far the optimal separation distance are interleaved lattice-based designs~\citep{Maximin}.

We begin by giving magic rotations for integer lattices. 
When $p=2^z$ with $z \in \mathbb{N}$, magic rotations we find for $\mathbb{Z}^p$ can be expressed as 
\[ R_z(V_1,\ldots,V_z,q_1,\ldots,q_z) = R_1(V_z,q_z) \otimes .s \otimes R_1(V_1,q_1), \]
where 
\begin{equation*}\label{eqn:R2}
R_1(V_l,q_l) = V_l Q(q_l) W(V_l,q_l),  
\end{equation*} 
\begin{equation*}\label{eqn:Q}
Q(q_l) = \left(\begin{array}{cc}
1 & 1 \\
-q_l^{1/2} & q_l^{1/2} 
\end{array}\right), 
\quad
W(V_l,q_l) = \left(\begin{array}{cc}
w_{1}(V_l,q_l) & 0 \\
0 & w_{2}(V_l,q_l)
\end{array}\right), 
\end{equation*} 
\begin{equation*}\label{eqn:w1}
w_{1}(V_l,q_l) = \left\{ \left(v_{l,1,1} - v_{l,1,2} q_l^{1/2}\right)^2 + \left(v_{l,2,1} - v_{l,2,2} q_l^{1/2}\right)^2 \right\}^{-1/2}, 
\end{equation*} 
\begin{equation*}\label{eqn:w2}
w_{2}(V_l,q_l) = \left\{ \left(v_{l,1,1} + v_{l,1,2} q_l^{1/2}\right)^2 + \left(v_{l,2,1} + v_{l,2,2} q_l^{1/2}\right)^2 \right\}^{-1/2},   
\end{equation*}
and $v_{l,i,j}$ denote the $(i,j)$th entry of the $2\times 2$ matrix $V_l$. 
Clearly, $R_z(V_1,\ldots,V_z,q_1,\ldots,q_z)$ is an orthogonal matrix if $q_l>0$, $V_l$ has full rank, and 
\begin{equation}\label{eqn:Vq}
v_{l,1,1}^2 + v_{l,2,1}^2 = q_l ( v_{l,1,2}^2 + v_{l,2,2}^2 ) 
\end{equation}
for $l=1,\ldots,z$. 
Theorem~\ref{thm:power} below verifies that these $R_z(V_1,\ldots,V_z,q_1,\ldots,q_z)$ are magic rotations for $\mathbb{Z}^p$. 
We remark that tensor product types of rotations have been employed to construct orthogonal Latin hypercube designs~\citep{Steinberg:Lin:2006}. 

\begin{theorem}\label{thm:power}
Assume that $z\in \mathbb{N}$, $v_{l,i,j} \in \mathbb{N}$, $q_l \in \mathbb{N}$, $V_l$ has full rank, and the elements of $V_1, \ldots, V_z$ satisfy (\ref{eqn:Vq}). 
Also assume that $\prod_{l=1}^z q_l^{i_l/2} $ is irrational for any $(i_1,\ldots,i_z) \in \{0,1\}^z$ and $(i_1,\ldots,i_z) \neq 0$. 
Then, designs generated from $L\{R_z(V_1,\ldots,V_z,q_1,\ldots,q_z)\}$ are quasi-Latin hypercube designs in $2^z$ dimensions that possess quasi-optimal separation distance on all of the projections. 
\end{theorem}

While proofs are deferred to the appendix, 
we write one important tool to verify the quasi-Latin hypercube property of lattice-based designs that is repeatedly used in proofs, as follows. 

\begin{proposition} \label{prp:minimax}
Suppose $L$ is a lattice and $L$-based designs possess quasi-optimal $\rho_{\{k\}}(D)$. 
Then, $L$-based designs also possess quasi-optimal $\eta_k(D)$. 
\end{proposition}

We use Minkowski's first theorem~\citep{Siegel:1989} below, a fundamental result in geometry of numbers, to prove Proposition~\ref{prp:minimax}.

\begin{lemma} [Minkowski's first theorem] 
Any convex set in $\mathbb{R}^p$ that is symmetric with respect to the origin and with volume greater than $|\det(G)|2^p$ contains a non-zero lattice point of $L(G)$.
\end{lemma}

Proposition~\ref{prp:minimax} states that all lattice-based designs that possess quasi-optimal separation distance on univariate projections also possess quasi-optimal fill distance on univariate projections and thus are quasi-Latin hypercube designs. 
From Theorem~\ref{thm:power}, we can generate quasi-Latin hypercube designs that possess quasi-optimal separation distance on all of the projections for arbitrary $p$ and $n$. 
Even if $p$ is not a power of 2, we can always find a $z\in\mathbb{N}$ such that $2^{z-1}<p\leq 2^z$, generate an $L\{R_z(V_1,\ldots,V_z,q_1,\ldots,q_z)\}$-based design, and project the design onto the first $p$ dimensions. 
Nevertheless, because $\mathbb{Z}^p$-based designs are sub-optimal in unprojected separation distance, it is more interesting to obtain magic rotations for densest packings, thinnest coverings, and interleaved lattices. 

The $L(G)$ is a subset of $L(H)$ if and only if there exists a $p\times p$ integer matrix $K$ such that $G=KH$ and $|\det(K)|\geq 1$.
In this case, we call $L(G)$ a sublattice of $L(H)$. 
Proposition~\ref{prp:extension} below gives the projection property of sublattices. 

\begin{proposition}\label{prp:extension} 
Suppose designs generated from $L(GR)$ possess quasi-optimal $\rho_\gamma(D)$ and $L(H)$ is a sublattice of $L(G)$.  
Then, designs generated from $L(HR)$ possess quasi-optimal $\rho_\gamma(D)$ as well.  
\end{proposition}

Utilizing Proposition~\ref{prp:extension}, we extend the results in Theorem~\ref{thm:power} to some sublattices of $\mathbb{Z}^p$. 

\begin{corollary} \label{crl:power}
(i) Assume that $v_{l,i,j} \in \mathbb{N}$, $q_l \in \mathbb{N}$, $V_l$ has full rank, and the elements of $V_1$, $V_2$, $V_3$, and $V_4$ satisfy (\ref{eqn:Vq}). 
Also assume that $q_1^{i_1/2}q_2^{i_2/2}q_3^{i_3/2}q_4^{i_4/2}$ is irrational for any $(i_1,i_2,i_3,i_4) \in \{0,1\}^4$ and $(i_1,i_2,i_3,i_4) \neq 0$. 
Then, designs generated from $L\{G_{\text{DP},4} R_2(V_1,V_2,q_1,q_2)\}$ in (\ref{eqn:DP24}), $L\{G_{\text{DP},8} R_3(V_1,V_2,V_3,q_1,q_2,q_3)\}$ in (\ref{eqn:DP24}),  $L\{G_{\text{DP},16} R_4(V_1,V_2,V_3,V_4,q_1,q_2,q_3,q_4)\}$ in (\ref{eqn:DP816}), 
or $L\{G_{\text{TC},8} R_3(V_1,V_2,V_3,q_1,q_2,q_3)\}$ in (\ref{eqn:TC}) are quasi-Latin hypercube designs 
that possess quasi-optimal separation distance on all of the projections. 
(ii) Assume $z\in \mathbb{N}$, $v_{l,i,j} \in \mathbb{N}$, $q_l \in \mathbb{N}$, $V_l$ has full rank, and the elements of $V_1, \ldots, V_z$ satisfy (\ref{eqn:Vq}). 
Also assume that $\prod_{l=1}^z q_l^{i_l/2} $ is irrational for any $(i_1,\ldots,i_z) \in \{0,1\}^z$ and $(i_1,\ldots,i_z) \neq 0$ and $L(G)$ is a standard interleaved lattice in $2^z$ dimensions. 
Then, designs generated from $L\{GR_z(V_1,\ldots,V_z,q_1,\ldots,q_z)\}$ are quasi-Latin hypercube designs that possess quasi-optimal separation distance on all of the projections. 
\end{corollary}

For illustration, when $z=2$, $q_1=2$, $q_2=5$, 
\[ V_1 = \left( \begin{array}{rr} 3 & 1 \\ 11 & 8 \end{array} \right), 
\text{ and } V_2 = \left( \begin{array}{rr} 10 & 1 \\ 15 & 8 \end{array} \right), \]
simple calculation yields 
\[ v_{1,1,1}^2 + v_{1,2,1}^2 = q_1 ( v_{1,1,2}^2 + v_{1,2,2}^2 ) = 130, \]
\[ v_{2,1,1}^2 + v_{2,2,1}^2 = q_2 ( v_{2,1,2}^2 + v_{2,2,2}^2 ) = 325, \]
\[ R_1(V_1,q_1) = \left( \begin{array}{cc} 3 - 2^{1/2} & 3 + 2^{1/2} \\ 11 - 8 \times 2^{1/2} & 11 + 8 \times 2^{1/2} \end{array} \right) \left( \begin{array}{cc} 0.619 & 0 \\ 0 & 0.044 \end{array} \right) 
 = \left( \begin{array}{rr}  0.981 &  0.194 \\
-0.194 &  0.981 \end{array} \right),\]
\[ R_1(V_2,q_2) = \left( \begin{array}{cc} 10 - 5^{1/2} & 10 + 5^{1/2} \\ 15 - 8 \times 5^{1/2} & 15 + 8 \times 5^{1/2} \end{array} \right) \left( \begin{array}{cc} 0.121 & 0 \\ 0 & 0.028 \end{array} \right) 
 = \left( \begin{array}{rr}  0.937 &  0.349 \\
-0.349 &  0.937
\end{array} \right),\]
and  
\[ R_2(V_1,V_2,q_1,q_2) 
= \left( \begin{array}{rrrr} 
 0.919 &  0.182 &  0.342 &  0.068 \\
-0.182 &  0.919 & -0.068 &  0.342 \\
-0.342 & -0.068 &  0.919 &  0.182 \\
 0.068 & -0.342 & -0.182 &  0.919 
\end{array} \right),
\] 
which is a magic rotation for both $\mathbb{Z}^4$ and $L(G_{\text{DP},4})$. 
Together with $G=G_{\text{DP},4}$ in (\ref{eqn:DP24}), $|\det(G)|=2$, $n=40$, and $\delta = (2.160, 1.505, 1.198, 0.820)^T$, 
we calculate the final design,  
whose design matrix and corresponding $a$ vectors are reported in appendix. 

Not all of the lattices are sublattices of an integer lattice. 
Theorem~\ref{thm:4} gives magic rotations for $L(G_{\text{TC},4})$, which cannot be expressed as a sublattice of $\mathbb{Z}^4$. 
Let 
\[ \bar V_1 = \left(\begin{array}{cc}
5v_{1,1,2} & v_{1,1,1} \\
5v_{1,2,2} & v_{1,2,1}
\end{array}\right) \] 
and $\bar B_4 = 5 V_2 \otimes V_1 - V_2 \otimes \bar V_1 - (J_2 V_2) \otimes (J_2 V_1)$. 

\begin{theorem}\label{thm:4}
Assume that $v_{l,i,j} \in \mathbb{N}$ for $i,j=1,2$ and $l=1,2$, $q_1=5$, $q_2 \in \mathbb{N}$, $q_2^{1/2}$ and $(5q_2)^{1/2}$ are irrational, $V_1$, $V_2$, and $\bar B_4$ have full rank, and 
the elements of $V_1$ and $V_2$ satisfy (\ref{eqn:Vq}). 
Then, designs generated from $L\{G_{\text{TC},4} R_2(V_1,V_2,q_1,q_2)\}$ are quasi-Latin hypercube designs that possess quasi-optimal separation distance on all of the projections.   
\end{theorem}

Our technique may be used to obtain magic rotations for other lattices but separate derivations are needed. 
We remark that all of the obtained magic rotations are for lattices in power of two dimensions.

\section{Densest packing-based maximum projection designs}
\label{sec:alg}

To corroborate the usefulness of the proposed magic rotations, 
in this section, we give an algorithm to generate densest packing-based maximum projection designs, 
which can be seen as a subclass of rotated sphere packing designs using $G=G_{\text{DP},p}$ in (\ref{eqn:DP24}) and (\ref{eqn:DP816}) and magic rotations rather than random rotations. 

First, we consider the case with $p=4$. 
Recall that we have found infinitely many magic rotations for the four-dimensional densest packing. 
From numerical results not shown here, it seems that lower $q_l$ tends to yield slightly better designs. 
Therefore, we use $q_1=2$ and $q_2=5$, which are the lowest integers that can satisfy (\ref{eqn:Vq}). 
We numerically find the 10 choices of $V_1$ and $V_2$ that have the lowest $\max_{i,j} v_{1,i,j}$ and $\max_{i,j} v_{2,i,j}$, respectively, while satisfying (\ref{eqn:Vq}). 
These choices are listed in Table~\ref{tab:V1V2}. 
The 100 magic rotations are obtained by crossing the choices of $V_1$ and $V_2$. 
We call the generated designs densest packing-based maximum projection designs. 

\begin{table}
\begin{center}
\caption{Choices of $V_1$, $q_1$, $V_2$, and $q_2$ for magic rotations in $p=4$}
\begin{tabular}{ccccc}
$v_{1,1,1}$ & $v_{1,1,2}$ & $v_{1,2,1}$ & $v_{1,2,2}$ & $q_1$ \\
1 & 1 & 3 & 2 & 2 \\
3 & 1 & 11 & 8 & 2 \\
7 & 6 & 11 & 7 & 2 \\
1 & 1 & 17 & 12 & 2 \\
11 & 6 & 17 & 13 & 2 \\
9 & 8 & 17 & 11 & 2 \\
4 & 1 & 18 & 13 & 2 \\
7 & 3 & 19 & 14 & 2 \\
3 & 4 & 19 & 13 & 2 \\
13 & 11 & 19 & 12 & 2 
\end{tabular}
\quad
\begin{tabular}{ccccc}
$v_{2,1,1}$ & $v_{2,1,2}$ & $v_{2,2,1}$ & $v_{2,2,2}$ & $q_2$ \\
1 & 1 & 3 & 1 & 5 \\
1 & 1 & 7 & 3 & 5 \\
6 & 1 & 7 & 4 & 5 \\
7 & 1 & 9 & 5 & 5 \\
8 & 1 & 11 & 6 & 5 \\
11 & 2 & 12 & 7 & 5 \\
10 & 1 & 15 & 8 & 5 \\
13 & 2 & 16 & 9 & 5 \\
11 & 1 & 17 & 9 & 5 \\
16 & 3 & 17 & 10 & 5 \\
\end{tabular}\label{tab:V1V2}
\end{center}
\end{table}

We compare the proposed designs with 
maximin distance Latin hypercube designs that are generated from the R package SLHD, 
maximum projection designs that are generated from the R package MaxPro, 
scrambled nets that are generated from the R package DiceDesign, 
and densest packing-based designs that are generated from 100 random rotations. 
For the last type of design, the algorithm as in generating original rotated sphere packing designs is used, except that the four-dimensional thinnest covering is replaced by the four-dimensional densest packing. 
These densest packing-based designs have comparable $\rho_\gamma(D)$ for $\text{card}(\gamma)<p$ and better $\rho(D)$ than original rotated sphere packing designs. 
Figure~\ref{fig:p4} displays the maximum projective separation distance, $\max_{\text{card}(\gamma)=r} \rho_{\gamma}(D)$, for $r=1,2,3,4$. 
We use the logarithmic scale on both $n$ and maximum projective separation distance so that the relationships are close to linear. 

\begin{figure}
\begin{center}
\includegraphics[width=7cm]{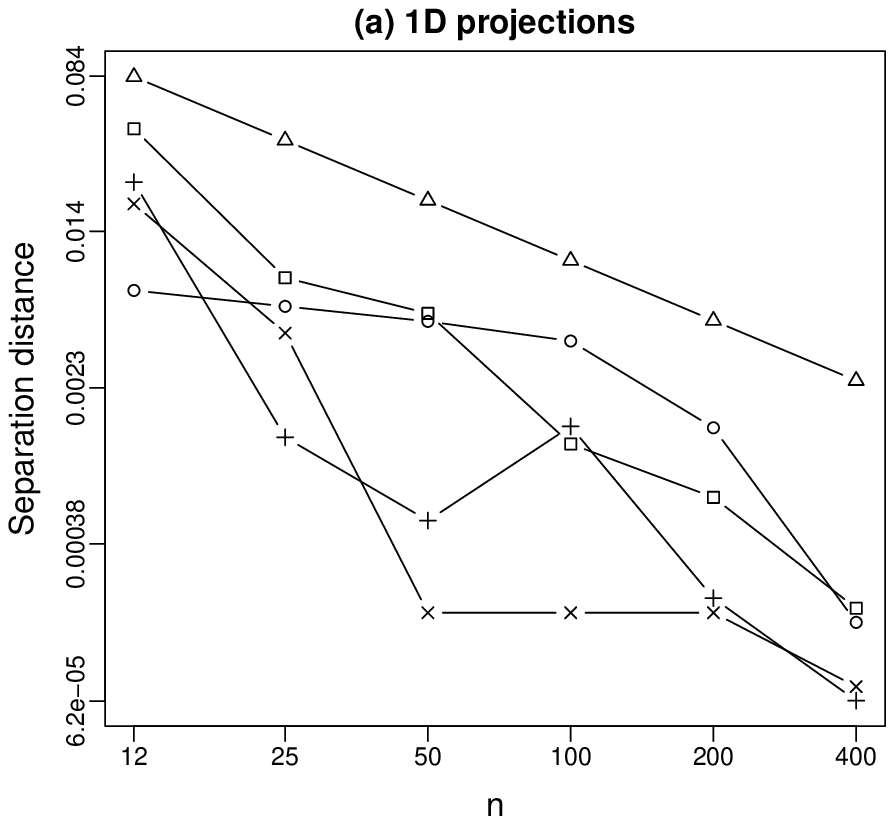}
\hspace{.2cm}
\includegraphics[width=7cm]{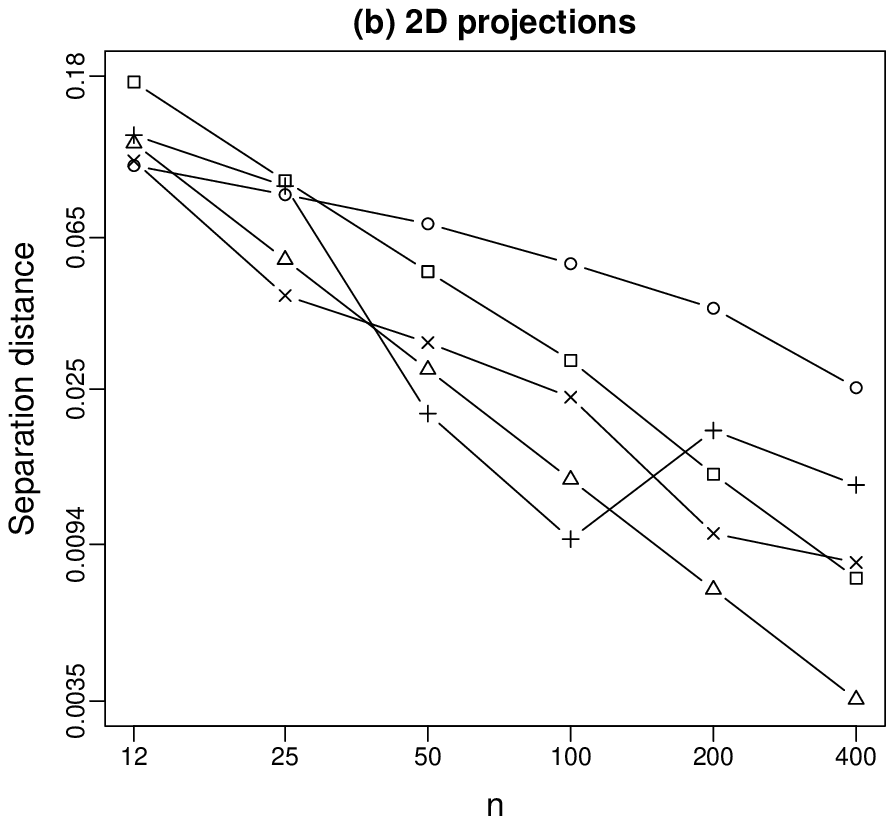}

\includegraphics[width=7cm]{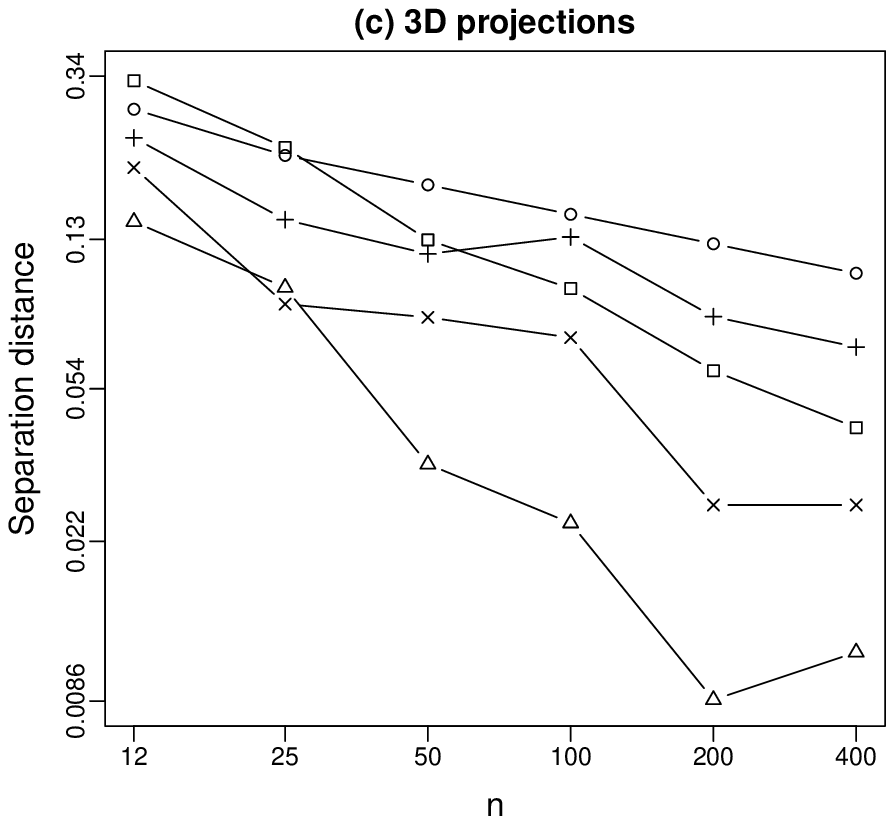}
\hspace{.2cm}
\includegraphics[width=7cm]{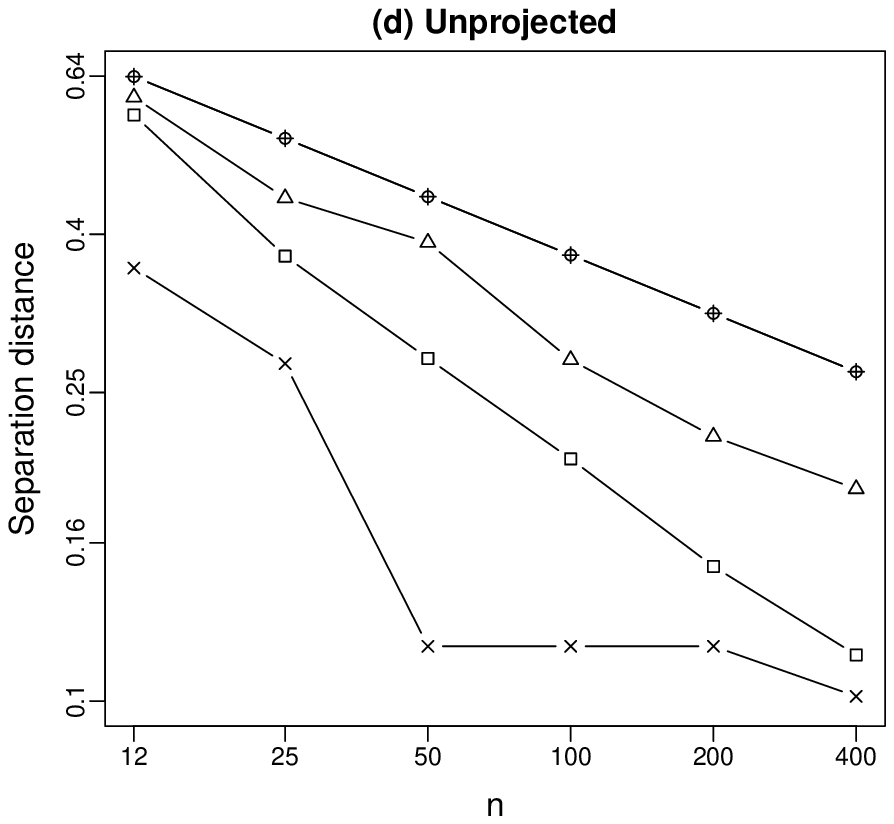}
\caption{Maximum projective separation distance in (a) one-, (b) two-, (c) three-, and (d) four-dimensional projections  
for densest packing-based maximum projection designs (circle), 
designs generated from random rotations (plus), 
maximin distance Latin hypercube designs (triangle), 
maximum projection designs (square), 
and scrambled nets (cross) 
in four dimensions. 
}
\label{fig:p4}
\end{center}
\end{figure}

From the results, four-dimensional densest packing-based maximum projection designs are remarkably better than other types of designs for moderate-to-large $n$. 
They are the best for $\max_{\text{card}(\gamma)=2} \rho_{\gamma}(D)$, $\max_{\text{card}(\gamma)=3} \rho_{\gamma}(D)$, and $\rho(D)$ when $n\geq 25$. 
For $\max_{\text{card}(\gamma)=1} \rho_{\gamma}(D)$, they are reasonably good, although not as good as maximin distance Latin hypercube designs which are optimal in univariate margins.
In our opinion, $\max_{\text{card}(\gamma)=1} \rho_{\gamma}(D)$ is important only if there is only one active variable. 
However, under this scenario any reasonably good design should be enough because modeling a single-variable function is a relatively simple task. 
Compared to densest packing-based maximum projection designs, maximin distance Latin hypercube designs are considerably worse in $\max_{\text{card}(\gamma)=2} \rho_{\gamma}(D)$ and 
$\max_{\text{card}(\gamma)=3} \rho_{\gamma}(D)$, 
maximum projection designs are considerably worse in $\rho(D)$, 
scrambled nets are considerably worse in all measures, 
and densest packing-based designs that are generated from random rotations are uniformly not better in any sense. 
We remark that from comparisons not shown here, interleaved lattice-based maximin distance designs~\citep{Maximin} have better $\rho(D)$ than densest packing-based maximum projection designs. 
However, their projective separation distances are poor. 

For $p=8$, we also recommend trying 100 magic rotations and then selecting the empirically best design by $\psi(D)$ in (\ref{eqn:MPC}). 
Since the third-lowest $q_l \in \mathbb{N}$ that may satisfy (\ref{eqn:Vq}) is 13, we use $q_1=2$, $q_2=5$, and $q_3=13$. 
We numerically find the 10 choices of $V_1$, $V_2$, and $V_3$ that have the lowest $\max_{i,j} v_{1,i,j}$, $\max_{i,j} v_{2,i,j}$, and $\max_{i,j} v_{3,i,j}$, respectively, while satisfying (\ref{eqn:Vq}). 
Crossing these choices yields 1000 magic rotations and we recommend trying 100 randomly selected rotations among them. 

\begin{figure}
\begin{center}
\includegraphics[width=7cm]{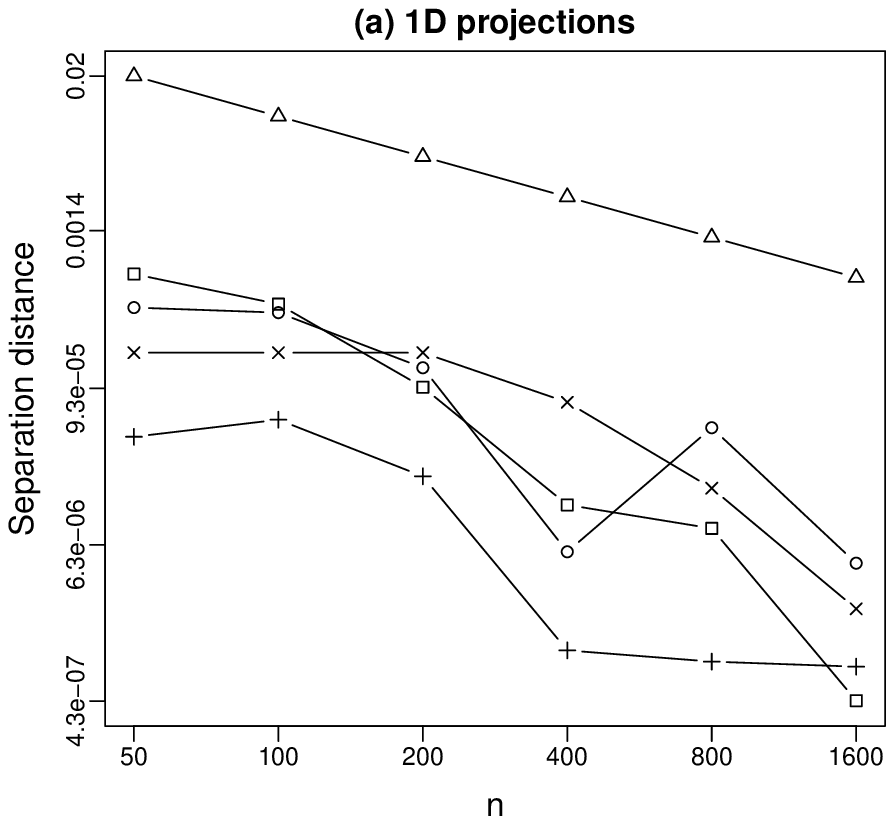}
\hspace{.2cm}
\includegraphics[width=7cm]{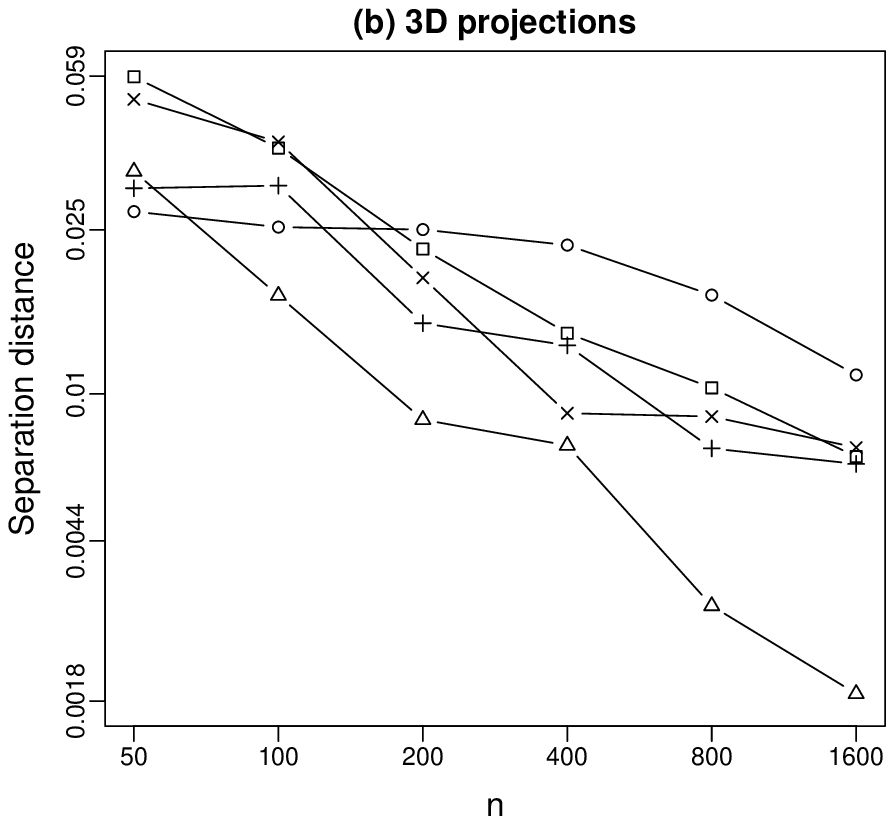}

\includegraphics[width=7cm]{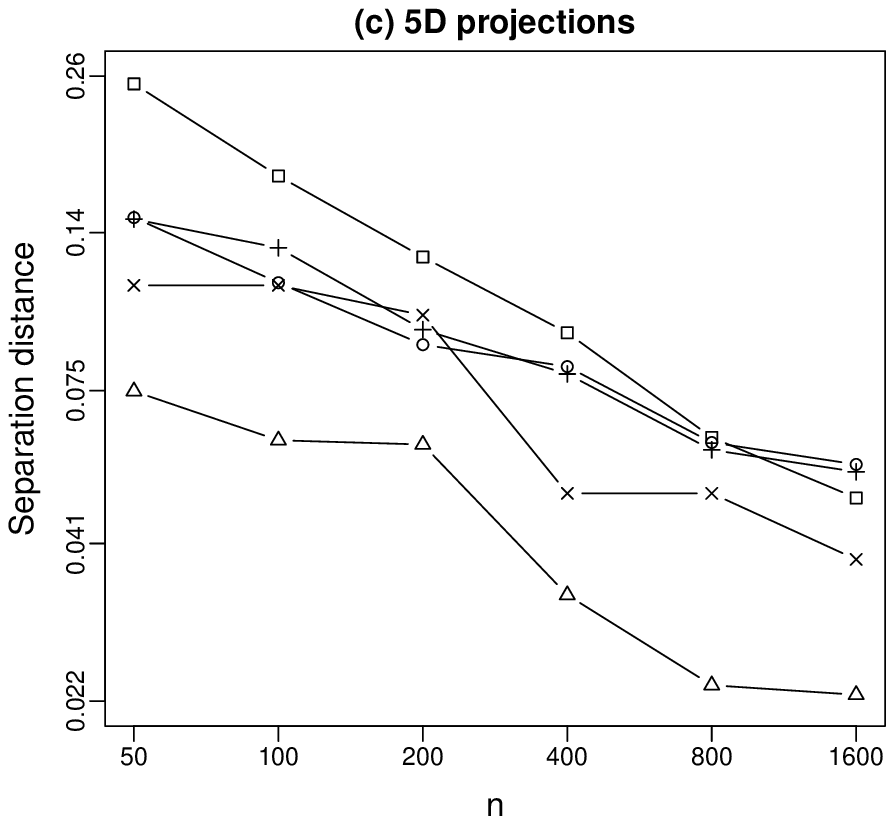}
\hspace{.2cm}
\includegraphics[width=7cm]{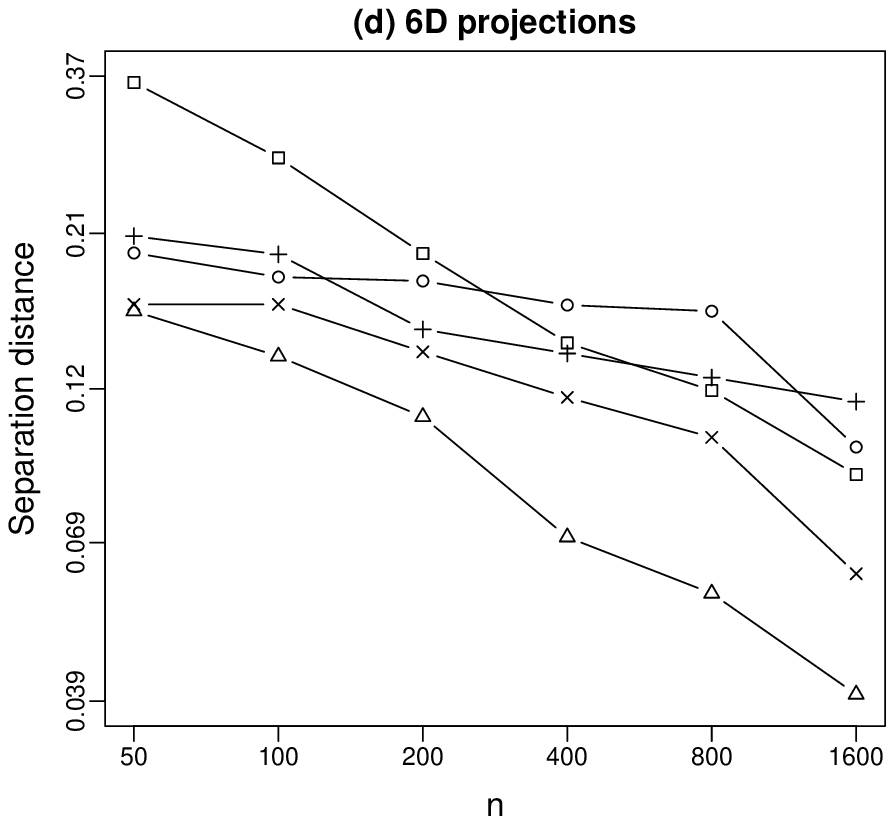}

\includegraphics[width=7cm]{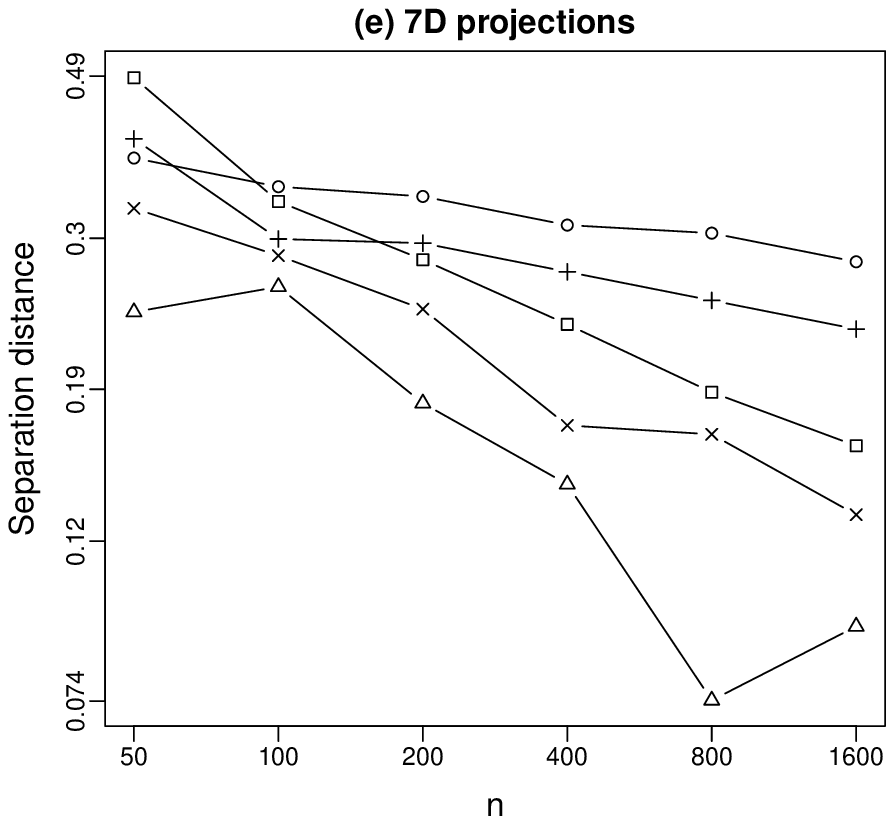}
\hspace{.2cm}
\includegraphics[width=7cm]{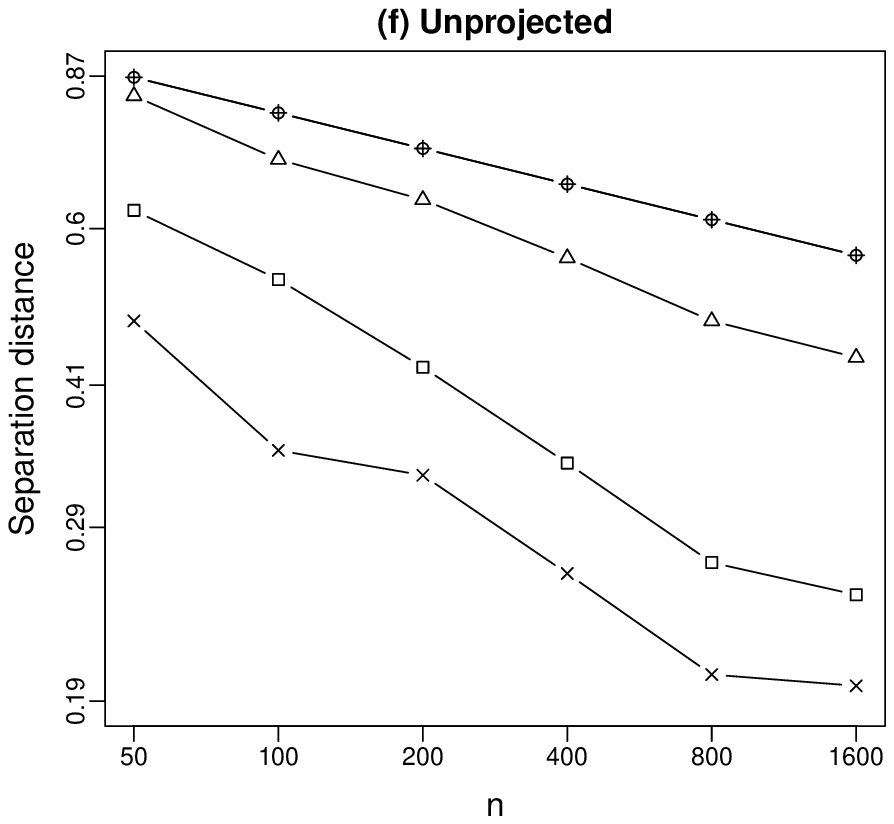}
\caption{Maximum projective separation distance in (a) one-, (b) three-, (c) five-, (d) six-, (e) seven-, and (f) eight-dimensional projections 
for densest packing-based maximum projection designs (circle), 
designs generated from random rotations (plus), 
maximin distance Latin hypercube designs (triangle), 
maximum projection designs (square), 
and scrambled nets (cross) 
in eight dimensions. 
}
\label{fig:p8a}
\end{center}
\end{figure}

Similar to the $p=4$ case, we numerically compare designs in eight dimensions and Fig.~\ref{fig:p8a} displays the results. 
Owing to space limitations, we omit the results for $\max_{\text{card}(\gamma)=2} \rho_{\gamma}(D)$ and $\max_{\text{card}(\gamma)=4} \rho_{\gamma}(D)$, which are similar to that for $\max_{\text{card}(\gamma)=3} \rho_{\gamma}(D)$. 
Although densest packing-based maximum projection designs in $p=8$ are not as striking as those in $p=4$, 
they are still uniformly better than other types of designs in $\max_{\text{card}(\gamma)=r} \rho_{\gamma}(D)$ for $r>1$ with large $n$. 
For moderate $n$, they are also excellent in overall performance: 
they are the best for $\rho(D)$, which is considerably better than maximum projection designs and scrambled nets;  
they are also the best for $\max_{\text{card}(\gamma)=7} \rho_{\gamma}(D)$, which is considerably better than maximin distance Latin hypercube designs and scrambled nets;  
they are reasonably good for $\max_{\text{card}(\gamma)=3} \rho_{\gamma}(D)$, $\max_{\text{card}(\gamma)=5} \rho_{\gamma}(D)$, and $\max_{\text{card}(\gamma)=6} \rho_{\gamma}(D)$, which are substantially better than maximin distance Latin hypercube designs;  
and, finally, they are uniformly no worse than densest packing-based designs that are generated from random rotations. 
In sum, they strike a new balance between projective separation distance and unprojected separation distance that is more favorable to $\rho(D)$ and $\max_{\text{card}(\gamma)=r} \rho_{\gamma}(D)$ for high $r$ 
and should be useful for computer experiments that likely contain six or more active variables. 

\begin{figure}
\begin{center}
\includegraphics[width=7cm]{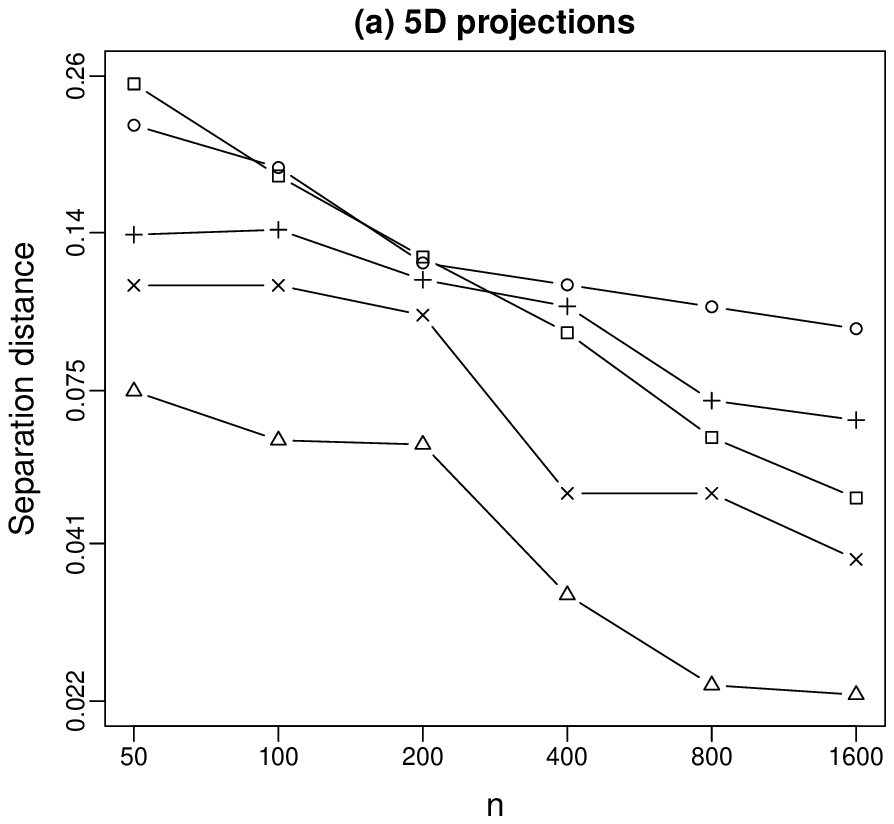}
\hspace{.2cm}
\includegraphics[width=7cm]{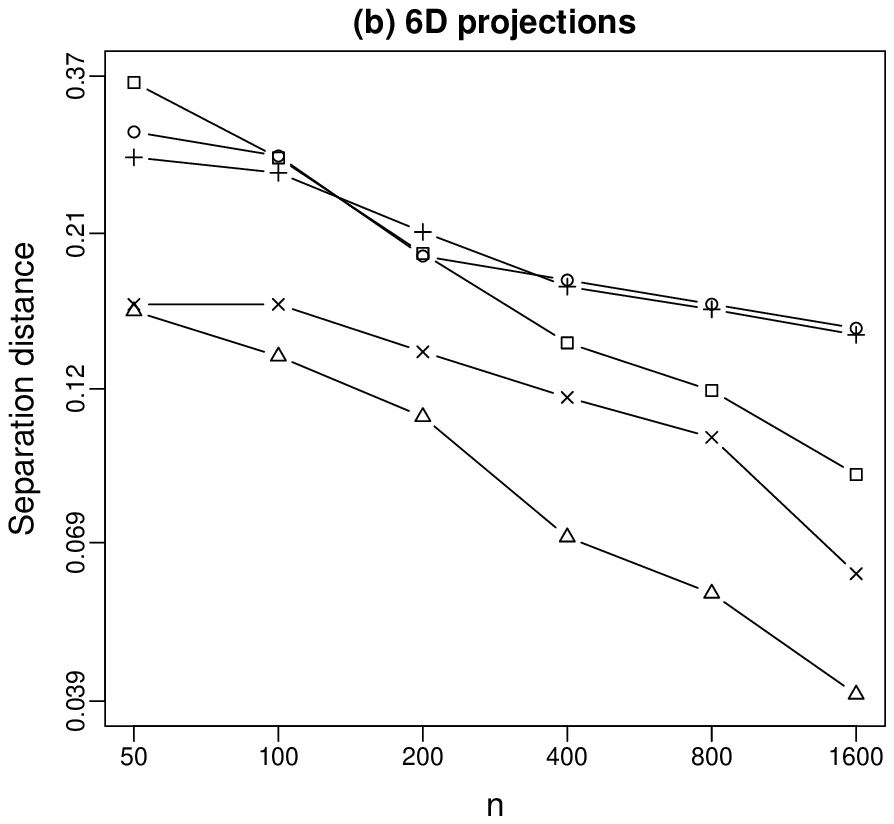}

\includegraphics[width=7cm]{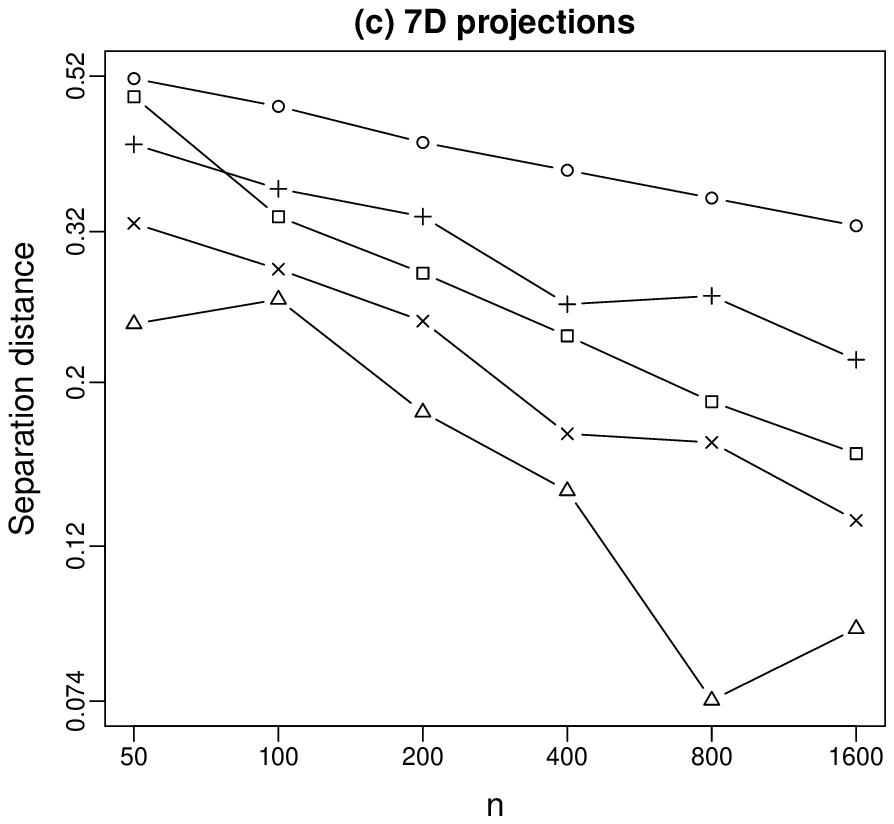}
\hspace{.2cm}
\includegraphics[width=7cm]{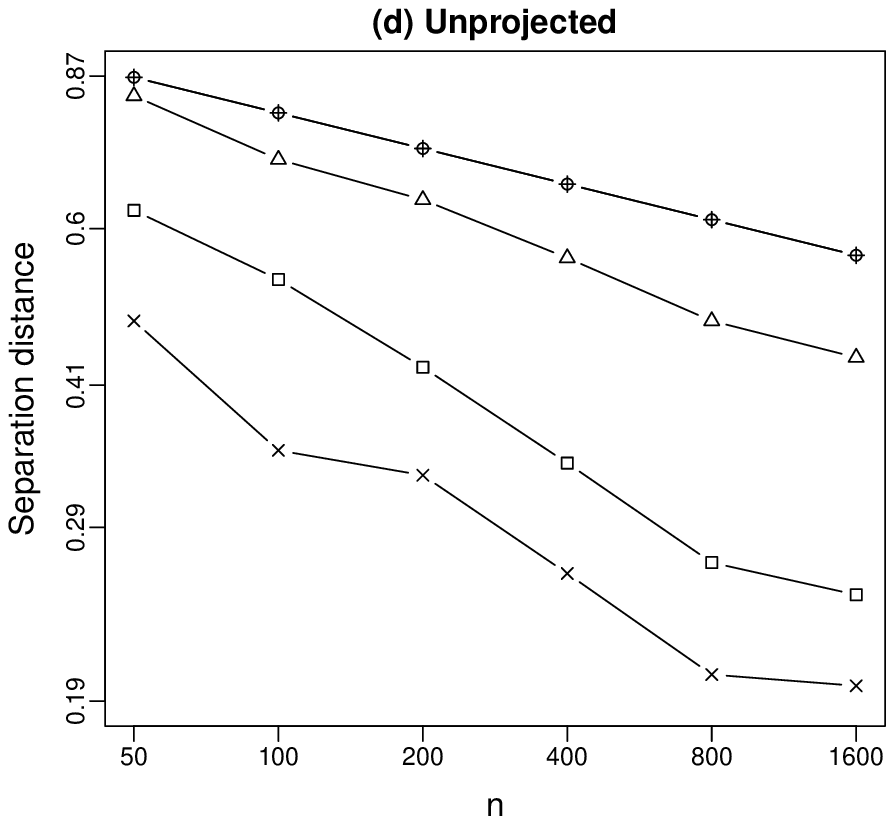}
\caption{Maximum projective separation distance in (a) five-, (b) six-, (c) seven-, and (d) eight-dimensional projections  
for densest packing-based maximum projection designs using $\varphi(D)$ (circle), 
designs generated from random rotations using $\varphi(D)$ (plus), 
maximin distance Latin hypercube designs (triangle), 
maximum projection designs (square), 
and scrambled nets (cross) 
in eight dimensions. 
}
\label{fig:p8b}
\end{center}
\end{figure}

Moreover, one advantage of densest packing-based maximum projection designs is that the methodology allows flexible score functions to select the empirically best design. 
For instance, when from prior knowledge we are certain that there are at least five active variables, the $\psi(D)$ criterion may not be optimal. 
A more specialized score function is $\varphi(D) = \sum_{r=5}^8 \log\{ \max_{\text{card}(\gamma)=r} \rho_{\gamma}(D) n^{1/r} \}$, 
which reflects the separation distance properties of designs projected onto five- and larger-dimensional subspaces. 
It is equivalent to $\sum_{r=5}^7 \log\{ \max_{\text{card}(\gamma)=r} \rho_{\gamma}(D) n^{1/r} \}$ for eight-dimensional densest packing-based designs because all of them have the same $\rho(D)$. 
Figure~\ref{fig:p8b} displays the numerical comparison results in $\max_{\text{card}(\gamma)=r} \rho_{\gamma}(D)$ with $r=5,6,7,8$, in which we use $\varphi(D)$ instead of $\psi(D)$ in Step~3 to select the final design. 
From the results, densest packing-based maximum projection designs are uniformly the best for moderate-to-large $n$. 
Similar to the $p=4$ case, maximin distance Latin hypercube designs are sub-optimal in $\max_{\text{card}(\gamma)=r} \rho_{\gamma}(D)$ with $r=5,6,7$, 
maximum projection designs are sub-optimal in $\rho(D)$, 
scrambled nets are uniformly no better than maximum projection designs, 
and densest packing-based designs generated from random rotations are uniformly no better than densest packing-based maximum projection designs. 
We thus conclude that densest packing-based maximum projection designs are uniformly optimal among the five types of designs provided that there are at least five active variables. 

One may argue that the comparison is unfair for maximum projection designs because they optimize $\psi(D)$, which accounts for both low- and high-dimensional projections. 
If we use the same numerical optimization technique for maximum projection designs but optimize $\varphi(D)$ instead of $\psi(D)$, we should generate designs much better in high-dimensional projections. 
However, as far as we know, from the algorithm the $\psi(D)$ criterion cannot be flexibly modified to account for prior knowledge on the level of sparsity. 
This is because the criterion needs to be repeatedly evaluated and therefore has to be very fast in computation. 
On the other hand, for densest packing-based designs we can afford to use a much slower criterion, such as $\varphi(D)$, because it will only be evaluated 100 times. 
Even using $\psi(D)$, it takes $7.3$ and $14.7$ minutes of one $2.70$GHz core of a laptop to generate a 400 point maximum projection design using default settings for $p=4$ and $p=8$, respectively,  
while it takes 1, 19, and 20 seconds to generate a 400 point densest packing-based maximum projection design for $p=4$ with $\psi(D)$, $p=8$ with $\psi(D)$, and $p=8$ with $\varphi(D)$, respectively.  

For $p=16$, the algorithm to generate densest packing-based maximum projection designs is much slower than that for $p=8$. 
To reduce computation, we recommend trying 10 magic rotations instead of 100. 
Since the fourth-lowest $q_l \in \mathbb{N}$ that may satisfy (\ref{eqn:Vq}) is 17, we recommend to use $q_1=2$, $q_2=5$, $q_3=13$, and $q_4=17$. 
We numerically find the 10 choices of $V_1$, $V_2$, $V_3$, and $V_4$ that have the lowest $\max_{i,j} v_{1,i,j}$, $\max_{i,j} v_{2,i,j}$, $\max_{i,j} v_{3,i,j}$, and $\max_{i,j} v_{4,i,j}$, respectively, while satisfying (\ref{eqn:Vq}). 
Crossing these choices yields 10000 magic rotations and we recommend trying 10 randomly selected rotations among them. 

Similar to the $p=4$ and $p=8$ cases, we numerically compare designs in sixteen dimensions. 
We only try 10 rotations for densest packing-based designs that are generated from random rotations because they are as slow as our proposed designs in construction. 
From the results, there is no method that is optimal in all projections. 
Maximin distance Latin hypercube designs are the best for $\max_{\text{card}(\gamma)=r} \rho_{\gamma}(D)$ with $r=1,2$, scrambled nets are the best for $\max_{\text{card}(\gamma)=r} \rho_{\gamma}(D)$ with $r=3,4,5$, maximum projection designs are the best for $\max_{\text{card}(\gamma)=r} \rho_{\gamma}(D)$ with $6\leq r\leq 13$, and densest packing-based maximum projection designs are the best for $\max_{\text{card}(\gamma)=r} \rho_{\gamma}(D)$ with $r=14, 15, 16$ and $n\geq 400$. 
Therefore, we conclude that densest packing-based maximum projection designs are useful when most variables are likely to be active. 
A figure displaying the numerical results for $p=16$ is provided in appendix. 

For $p=3$, theoretically we can generate a four-dimensional densest packing-based maximum projection design and then project it onto $\gamma = \{1,2,3\}$. 
However, from numerical results not provided here, we find that these designs are in general no better than designs generated from the three-dimensional densest packing with random rotations unless $n$ is very large. 
The case for other $p$ that are not a power of 2 is similar. 
This is a major limitation of the proposed methodology. 
Planned future work is to derive magic rotations for lattices in other $p$. 
While here we concentrate on densest packing-based designs, 
magic rotations for other lattices may also be useful in constructing space-filling designs.

\appendix

\section{Proofs}
\label{sec:proof}

\begin{proof}[Proof of Proposition~1]
Without loss of generality assume $L=L(GR)$, $k=p$, $|\det(G)|=1$, 
$\rho_{\{p\}}\{D(G,R,n,\delta)\} \geq c n^{-1}$, 
and $c \leq 1$. 
Namely, for sufficiently large $m$, 
\begin{equation}\label{eqn:c1}
L(GR) \cap \left( [-m,m]^{p-1} \times [0,c_1 m^{-p+1}] \right) = \emptyset, 
\end{equation}
where $c_1 = 2^{-p+1} c \leq 2^{-p+1}$. 
It suffices to show the existence of a $d\in\mathbb{R}$ such that for any $n>2 d$, $\delta\in\mathbb{R}^p$, and $z\in[0,1]$, there is at least one design point $y \in D(G,R,n,\delta)$ 
such that $y \in [0,1]^{p-1} \times [\max(0,z- d n^{-1}),\min(1,z+d n^{-1})]$. 
Without loss of generality assume $z\leq 1/2$.

Let $\tilde y = n^{-1/p} (\tilde a^T G +\delta)$ be the point that has the highest $p$th dimensional value among the points in 
$ \{ n^{-1/p} (a^T G +\delta) : a \in \mathbb{Z}^p \} \cap \left( [0,1]^{p-1} \times (-\infty,z) \right)$. 
Clearly such $\tilde a \in \mathbb{Z}^p$ exists and is unique. 
Without loss of generality assume $\tilde y \in [0,1/2]^{p-1} \times (-\infty,z)$. 
If there is an $a\in\mathbb{Z}^p$ such that $a^TG \in [0,n^{1/p}/2]^{p-1} \times (0, d n^{1/p} n^{-1})$, then $n^{-1/p} \{(\tilde a^T+a) G +\delta\} \in [0,1]^{p-1} \times [\max(0,z-d n^{-1}),\min(1,z+d n^{-1})]$. 
Consequently, it suffices to show the existence of a $c_2>0$ such that the set $L(GR) \cap ( [0,m]^{p-1} \times (0,c_2 m^{-p+1}] )$ is not empty for for sufficiently large $m$. 
In the rest of the proof we define $c_2$ and find an element of the set. 

When $p=2$, from Minkowski's first theorem, there exists an $x^{(1)}$ such that 
\[ x^{(1)} \in L(GR) \cap \left( [-c_1m,c_1m] \times [-1/(c_1m),1/(c_1m)] \right) \]
and $x^{(1)} \neq 0$. 
Considering $x^{(1)}$ and $-x^{(1)}$, there exists an $x^{(2)}$ such that 
\[ x^{(2)} \in L(GR) \cap \left( [-c_1m,c_1m] \times [0,1/(c_1m)] \right). \]

From (\ref{eqn:c1}), 
\[ L(GR) \cap \left( [-c_1m,c_1m] \times [0,1/m] \right) = \emptyset. \] 
Therefore, 
\[ x^{(2)} \in L(GR) \cap \left( [-c_1m,c_1m] \times (1/m,1/(c_1m)] \right). \]

From Minkowski's first theorem, there exists an $x^{(3)}$ such that 
\[ x^{(3)} \in L(GR) \cap \left( [-m,m] \times [-1/m,1/m] \right). \]
Considering $x^{(3)}$ and $-x^{(3)}$, there exists an $x^{(4)}$ such that 
\[ x^{(4)} \in L(GR) \cap \left( [-m,m] \times [0,1/m] \right). \]
From (\ref{eqn:c1}), 
\[ x^{(4)} \in L(GR) \cap \left\{ ( [-m,-c_1m) \cup (c_1m,m] ) \times (0,1/m] \right\}. \]

Let $c_2 = 1/c_1$.  
Considering $x^{(2)}$, $x^{(4)}$ and $x^{(2)}-x^{(4)}$,  
there exists a $y$ such that 
\[ y \in L(GR) \cap \left( [0,m] \times (0,c_2m^{-1}] \right). \]

When $p>2$, let 
\begin{eqnarray*}
b_2 &=& \left[ 1 + c_1^{p} \{2(p-2)\}^{-p(p-2)} \right]^{-1} m, \\
b_3 &=& c_1^{p} \{2(p-2)\}^{-(p-1)^2} b_2, \\
b_1 &=& c_1^{p-1} \{2(p-2)\}^{-(p-1)(p-2)} b_2, \\
\epsilon_1 &=& c_1^{-p^2+p+1} \{2(p-2)\}^{p(p-1)(p-2)} b_2^{-(p-1)} = c_1 \{2(p-2)b_3\}^{-(p-1)}, \\
\epsilon_2 &=& c_1^{-p^2+2p} \{2(p-2)\}^{(p-1)^2(p-2)} b_2^{-(p-1)} = c_1 b_1^{-(p-1)}. 
\end{eqnarray*}
Then $b_3 < 2(p-2)b_3 \leq b_1 < b_2$, $\epsilon_1>\epsilon_2$, and 
\begin{equation}\label{eqn:x1}
(2\epsilon_1) (2b_1) (2b_3)^{p-2} = (2\epsilon_2) (2b_2) (2b_3)^{p-2} = 2^p. 
\end{equation}

From Minkowski's first theorem and (\ref{eqn:x1}), there exists an $x^{(1)}$ such that 
\[ x^{(1)} \in L(GR) \cap \left( [-b_1,b_1] \times [-b_3,b_3]^{p-2} \times [-\epsilon_1,\epsilon_1] \right) \]
and $x^{(1)} \neq 0$. 
Considering $x^{(1)}$ and $-x^{(1)}$, there exists an $x^{(2)}$ such that 
\[ x^{(2)} \in L(GR) \cap \left( [-b_1,b_1] \times [-b_3,b_3]^{p-2} \times [0,\epsilon_1] \right). \]

From (\ref{eqn:c1}), 
\[ L(GR) \cap ( [-b_1,b_1]^{p-1} \times [0,\epsilon_2] ) = \emptyset. \] 
Therefore, 
\[ x^{(2)} \in L(GR) \cap \left( [-b_1,b_1] \times [-b_3,b_3]^{p-2} \times (\epsilon_2,\epsilon_1] \right). \]

From Minkowski's first theorem and (\ref{eqn:x1}), there exists an $x^{(3)}$ such that 
\[ x^{(3)} \in L(GR) \cap \left( [-b_2,b_2] \times [-b_3,b_3]^{p-2} \times [-\epsilon_2,\epsilon_2] \right). \]
Considering $x^{(3)}$ and $-x^{(3)}$, there exists an $x^{(4)}$ such that 
\[ x^{(4)} \in L(GR) \cap \left( [-b_2,b_2] \times [-b_3,b_3]^{p-2} \times [0,\epsilon_2] \right). \]
From (\ref{eqn:c1}), 
\[ x^{(4)} \in L(GR) \cap \left\{ ( [-b_2,-b_1) \cup (b_1,b_2] ) \times [-b_3,b_3]^{p-2} \times (0,\epsilon_2] \right\}. \]

Considering $x^{(2)}$, $x^{(4)}$ and $x^{(2)}-x^{(4)}$, 
there exists an $x^{(5)}$ such that 
\[ x^{(5)} \in L(GR) \cap \left( [0,b_2] \times [-2b_3,2b_3]^{p-2} \times (0,\epsilon_1] \right). \]

From (\ref{eqn:c1}), 
\[ L(GR) \cap \left\{ [-2(p-2)b_3,2(p-2)b_3]^{p-1} \times [0,\epsilon_1] \right\} = \emptyset. \] 
Therefore, 
\[ x^{(5)} \in L(GR) \cap \left\{ (2(p-2)b_3,b_2] \times [-2b_3,2b_3]^{p-2} \times (0,\epsilon_1] \right\}. \]

Let $y^{(1)}= x^{(5)}$. 
Similarly, for any $k=2,\ldots,p-1$, there exists a $y^{(k)}\in L(GR)$ such that 
$y^{(k)}_p \in (0,\epsilon_1]$, $y^{(k)}_k \in (2(p-2)b_3,b_2]$ and $y^{(k)}_j \in [-2b_3,2b_3]$ for any $j\neq p$, $j\neq k$. 
Let $y = \sum_{j=1}^{p-1} y^{(j)}$. 
Then 
\[ y \in L(GR) \cap \left\{ (0,b_2+2(p-2)b_3]^{p-1} \times (0,(p-1)\epsilon_1] \right\}. \]
Let $ c_2 = (p-1) c_1^{-p^2+p+1} \{2(p-2)\}^{p(p-1)(p-2)} [ 1 + c_1^{p} \{2(p-2)\}^{-p(p-2)} ]^{p-1} $ 
$ = (p-1)\epsilon_1 m^{p-1}$. 
Because $b_2+2(p-2)b_3=m$, 
we have 
\[ y \in L(GR) \cap \left( (0,m]^{p-1} \times (0, c_2 m^{-(p-1)} ] \right). \]
Combining the two cases of $p$ completes the proof. 
\end{proof}

\begin{proof}[Proof of Theorem~1]
Let $f(a) = (f_1(a),\ldots,f_p(a)) = a^T R_z(V_1,\ldots,V_z,q_1,\ldots,q_z)$. 
Clearly, $R_z(V_1$, $\ldots,V_z,q_1,\ldots,q_z) = (V_z \otimes .s \otimes V_1) \{Q(q_z) \otimes .s \otimes Q(q_1)\} \{W(V_z,q_z) \otimes .s \otimes W(V_1,q_1)\}$. 
Here $V_z \otimes .s \otimes V_1$ is an integer matrix with full rank, 
$W(V_z,q_z) \otimes .s \otimes W(V_1,q_1)$ is a positive diagonal matrix, 
and the entries in each column of $Q(q_z) \otimes .s \otimes Q(q_1)$ are rationally independent~\citep{Besicovitch}. 
As a result, for any $k=1,\ldots,p$, $f_k(a) \neq 0$ for any nonzero $a\in \mathbb{Z}^p$. 
On the other hand, let $w = \prod_{l=1}^z \{ w_1(V_l,q_l) w_2(V_l,q_l) \}^{-p}$, it is not hard to see that $w \prod_{k=1}^p f_k(a)$ is a polynomial of $a$ with integer coefficients. 
Therefore, $w \prod_{k=1}^p f_k(a) \in \mathbb{Z}$ for any $a\in \mathbb{Z}^p$. 
As a result, $|w \prod_{k=1}^p f_k(a)| \geq 1$ for any nonzero $a\in \mathbb{Z}^p$.

Now consider any $f(\bar a), f(\tilde a) \in \prod_{k=1}^p[-\delta_k,n^{1/p}-\delta_k]$ and $\bar a\neq \tilde a$. 
Clearly, $|f_k(\bar a)-f_k(\tilde a)| \leq n^{1/p}$ for any $k$ 
and $\prod_{k=1}^p |f_k(\bar a)-f_k(\tilde a)| \geq w^{-1}$. 
Therefore, for any $\emptyset \neq \gamma \subset \{1,\ldots,p\}$, 
$\prod_{k\in \gamma} |f_k(\bar a)-f_k(\tilde a)| \geq w^{-1} n^{(-p+|\gamma|)/p}$. 
Therefore, 
\[ \left\{ \sum_{k \in \gamma} \left[ n^{-1/p} \{f_k(\bar a)-f_k(\tilde a)\} \right]^2 \right\}^{1/2}
  \geq |\gamma|^{1/2} w^{-1/|\gamma|} n^{-1/|\gamma|}. \]
This concludes that designs generated from $L\{R_z(V_1,\ldots,V_z,q_1,\ldots,q_z)\}$ possess quasi-optimal separation distance on all of the projections. 
From Proposition~1, designs generated from $L\{R_z(V_1,\ldots,V_z,q_1,\ldots,q_z)\}$ also possess quasi-optimal fill distance on univariate projections. 
\end{proof}

\begin{proof}[Proof of Proposition~2]
Let $f(a) = (f_1(a),\ldots,f_p(a)) = a^T GR$. 
There exists a $c >0$ such that for any $f(\bar a), f(\tilde a) \in \prod_{k=1}^p[-\delta_k,\{n|\det(G)|\}^{1/p}-\delta_k]$, $\bar a\neq \tilde a$, $n\geq 2$, and $\delta \in \mathbb{R}^p$, 
\[ \left\{ \sum_{k \in \gamma} \left[ \{n|\det(G)|\}^{-1/p} \{f_k(\bar a)-f_k(\tilde a)\} \right]^2 \right\}^{1/2}
  \geq c n^{-1/|\gamma|}. \]

Let 
$s = |\det(H)|/|\det(G)| \in \mathbb{N}$. 
Then for any $f(\bar a), f(\tilde a) \in \prod_{k=1}^p[-\delta_k,\{n|\det(H)|\}^{1/p}-\delta_k]$, $\bar a\neq \tilde a$, $n>0$, and $\delta \in \mathbb{R}^p$, 
we have 
\[ \left\{ \sum_{k \in \gamma} \left[ \{n|\det(H)|\}^{-1/p} \{f_k(\bar a)-f_k(\tilde a)\} \right]^2 \right\}^{1/2}
  \geq c s^{-1/|\gamma|} n^{-1/|\gamma|}. \]
\end{proof}

\begin{proof}[Proof of Corollary~1]
The statements hold because $L\{G_{\text{DP},4}R_2(V_1,V_2,q_1,q_2)\} \subset L\{R_2(V_1,V_2,q_1,q_2)\}$, $L\{G_{\text{DP},8}R_3(V_1,V_2,V_3,q_1,q_2,q_3)\} \subset L\{R_3(V_1,V_2,V_3,q_1,q_2,q_3)\}$, $L\{G_{\text{DP},16}R_4(V_1,V_2,V_3,V_4,q_1,q_2,q_3,q_4)\} \subset L\{R_4(V_1,V_2,V_3,V_4,q_1,q_2,q_3,q_4)\}$, $L\{G_{\text{TC},8}R_3(V_1,V_2,V_3,q_1,q_2,q_3)\} \subset L\{R_3(V_1,V_2,V_3,q_1,q_2,q_3)\}$, and $L\{G R_z(V_1,\ldots,V_z,q_1,\ldots,q_z)\} \subset L\{R_z(V_1,\ldots,V_z,q_1,\ldots,q_z)\}$. 
\end{proof}

\begin{proof}[Proof of Theorem~2]
Let $f(a) = (f_1(a),\ldots,f_4(a)) = a^T G_{\text{TC},4}R_2(V_1,V_2,q_1,q_2)$. 
We have 
\[ G_{\text{TC},4}R_2(V_1,V_2,q_1,q_2) = \bar B_4 \{Q(q_2) \otimes Q(5)\} \{W(V_2,q_2) \otimes W(V_1,5)\} .\] 
Here $\bar B_4$ is an integer matrix with full rank, 
$W(V_2,q_2) \otimes W(V_1,5)$ is a positive diagonal matrix, 
and each column of $Q(q_2) \otimes Q(5)$ consists of $\{1,5^{1/2},q_2^{1/2},5^{1/2}q_2^{1/2}\}$ or their opposite numbers. 
As a result, for any $k=1,\ldots,4$, $f_k(a) \neq 0$ for any nonzero $a\in \mathbb{Z}^4$. 
On the other hand, let $w = \prod_{l=1}^2 \{ w_1(V_l,q_l) w_2(V_l,q_l) \}^{-4}$, it is not hard to see that $w \prod_{k=1}^4 f_k(a)$ is a polynomial of $a$ with integer coefficients. 
Therefore, $w \prod_{k=1}^4 f_k(a) \in \mathbb{Z}$ for any $a\in \mathbb{Z}^4$. 
As a result, $|w \prod_{k=1}^4 f_k(a)| \geq 1$ for any nonzero $a\in \mathbb{Z}^4$.

Now consider any $f(\bar a), f(\tilde a) \in \prod_{k=1}^4[-\delta_k,n^{1/4}-\delta_k]$ and $\bar a\neq \tilde a$. 
Clearly, $|f_k(\bar a)-f_k(\tilde a)| \leq n^{1/4}$ for any $k$ 
and $\prod_{k=1}^4 |f_k(\bar a)-f_k(\tilde a)| \geq w^{-1}$. 
Therefore, for any $\emptyset \neq \gamma \subset \{1,\ldots,4\}$, 
$\prod_{k\in \gamma} |f_k(\bar a)-f_k(\tilde a)| \geq w^{-1} n^{(-4+|\gamma|)/4}$. 
Therefore, 
\[ \left\{ \sum_{k \in \gamma} \left[ n^{-1/4} \{f_k(\bar a)-f_k(\tilde a)\} \right]^2 \right\}^{1/2}
  \geq |\gamma|^{1/2} w^{-1/|\gamma|} n^{-1/|\gamma|}. \]
This concludes that designs generated from $L\{G_{\text{TC},4}R_2(V_1,V_2,q_1,q_2)\}$ possess quasi-optimal separation distance on all of the projections. 
From Proposition~1, designs generated from $L\{G_{\text{TC},4}R_2(V_1,V_2,q_1,q_2)\}$ also possess quasi-optimal fill distance on univariate projections.
\end{proof}

\section{A design matrix}
\label{sec:figure}

The design matrix of the design introduced in the example can be expressed as $ M = \{n|\det(G)|\}^{-1/4} ( AGR + \delta^T )$, 
where 
{\scriptsize
\[ A = \left(\begin{array}{rrrr}
0 & 0 & 0 & 0 \\
1 & 1 & 0 & -1 \\
1 & 0 & 1 & -1 \\
0 & 1 & 1 & -1 \\
1 & 0 & 0 & 0 \\
0 & 1 & 0 & 0 \\
0 & 0 & 1 & 0 \\
-1 & 1 & 0 & 0 \\
-1 & 0 & 1 & 0 \\
-1 & 0 & 0 & 1 \\
-1 & -1 & 0 & 1 \\
0 & -1 & 1 & 0 \\
0 & -1 & 0 & 1 \\
1 & -1 & 0 & 0 \\
1 & -1 & 1 & 0 \\
-1 & 0 & -1 & 1 \\
0 & 1 & -1 & 0 \\
0 & 0 & -1 & 1 \\
-1 & 1 & -1 & 1 \\
1 & 0 & -1 & 0 \\
1 & 1 & -2 & 0 \\
1 & 1 & -1 & 0 \\
0 & -1 & -1 & 1 \\
1 & -1 & -1 & 1 \\
-1 & -1 & -1 & 2 \\
-1 & 0 & 0 & 0 \\
0 & 1 & 0 & -1 \\
0 & 0 & 1 & -1 \\
0 & 1 & 1 & -2 \\
1 & 0 & 0 & -1 \\
1 & 1 & 1 & -2 \\
0 & -1 & 0 & 0 \\
1 & -2 & 0 & 0 \\
-1 & -1 & 1 & 0 \\
0 & 0 & -1 & 0 \\
1 & 1 & -1 & -1 \\
1 & 0 & -2 & 0 \\
-1 & 1 & -1 & 0 \\
-1 & -1 & -1 & 1 \\
1 & -1 & -1 & 0 
\end{array}\right)
\text{ and }
M = \left(\begin{array}{rrrr}
0.454 & 0.473 & 0.532 & 0.546 \\
0.701 & 0.841 & 0.623 & 0.683 \\
0.647 & 0.511 & 0.953 & 0.629 \\
0.279 & 0.758 & 0.816 & 0.721 \\
0.784 & 0.420 & 0.585 & 0.876 \\
0.416 & 0.666 & 0.448 & 0.968 \\
0.362 & 0.336 & 0.778 & 0.914 \\
0.086 & 0.720 & 0.395 & 0.638 \\
0.032 & 0.390 & 0.725 & 0.584 \\
0.169 & 0.298 & 0.356 & 0.831 \\
0.208 & 0.105 & 0.440 & 0.409 \\
0.401 & 0.143 & 0.862 & 0.492 \\
0.538 & 0.051 & 0.493 & 0.739 \\
0.822 & 0.227 & 0.669 & 0.454 \\
0.731 & 0.090 & 0.915 & 0.822 \\
0.261 & 0.435 & 0.110 & 0.462 \\
0.508 & 0.803 & 0.202 & 0.599 \\
0.591 & 0.381 & 0.163 & 0.792 \\
0.223 & 0.628 & 0.026 & 0.884 \\
0.876 & 0.557 & 0.339 & 0.508 \\
0.929 & 0.887 & 0.008 & 0.561 \\
0.838 & 0.750 & 0.255 & 0.929 \\
0.629 & 0.188 & 0.247 & 0.371 \\
0.959 & 0.135 & 0.300 & 0.701 \\
0.345 & 0.013 & 0.072 & 0.655 \\
0.124 & 0.527 & 0.478 & 0.216 \\
0.371 & 0.895 & 0.570 & 0.353 \\
0.317 & 0.565 & 0.900 & 0.299 \\
0.234 & 0.987 & 0.938 & 0.106 \\
0.739 & 0.648 & 0.707 & 0.261 \\
0.564 & 0.933 & 0.992 & 0.436 \\
0.492 & 0.280 & 0.615 & 0.124 \\
0.861 & 0.034 & 0.752 & 0.032 \\
0.071 & 0.197 & 0.808 & 0.162 \\
0.546 & 0.610 & 0.285 & 0.178 \\
0.792 & 0.978 & 0.377 & 0.315 \\
0.968 & 0.694 & 0.092 & 0.139 \\
0.178 & 0.857 & 0.148 & 0.269 \\
0.299 & 0.242 & 0.193 & 0.041 \\
0.914 & 0.364 & 0.422 & 0.086 
\end{array}\right). \]
}

\section{Numerical results on 16-dimensional designs}

Figure~\ref{fig:p16} below gives some numerical results comparing 16-dimensional designs. 

\begin{figure}
\begin{center}
\includegraphics[width=7cm]{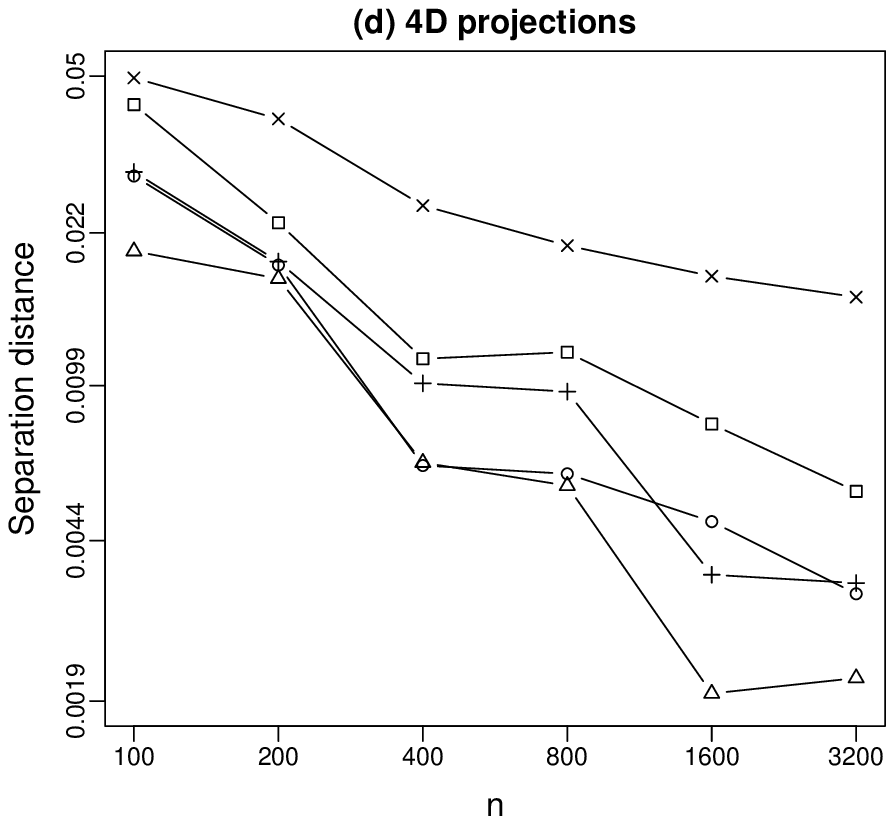}
\hspace{.2cm}
\includegraphics[width=7cm]{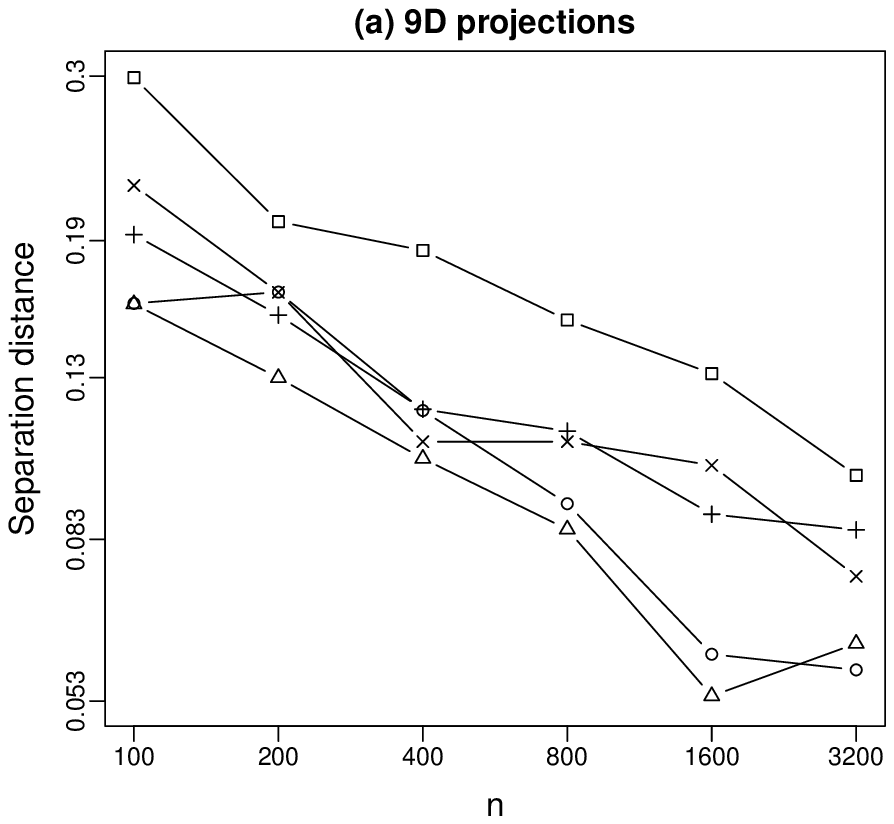}

\includegraphics[width=7cm]{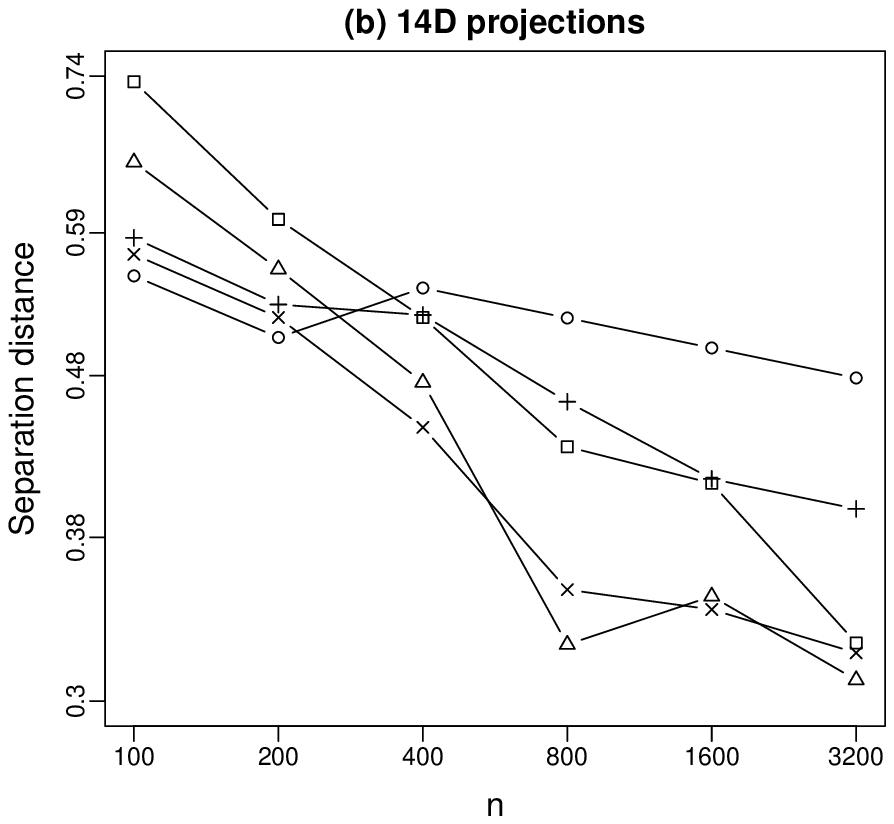}
\hspace{.2cm}
\includegraphics[width=7cm]{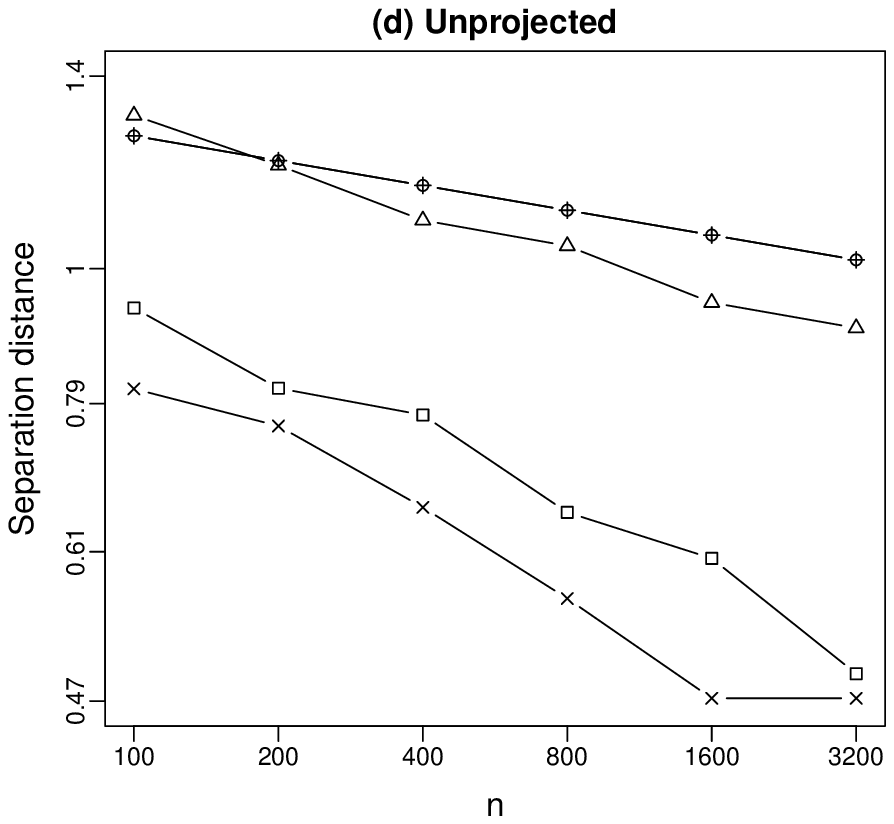}
\caption{Maximum projective separation distance in (a) four-, (b) nine-, (c) fourteen-, and (d) sixteen-dimensional projections  
for densest packing-based maximum projection designs (circle), 
designs generated from random rotations (plus), 
maximin distance Latin hypercube designs (triangle), 
maximum projection designs (square), 
and scrambled nets (cross) 
in 16 dimensions. 
}
\label{fig:p16}
\end{center}
\end{figure}

\bibliographystyle{Chicago}
\bibliography{QuasiLHDBiomk}

\end{document}